\documentclass[12pt,twoside]{amsart}

\usepackage{amsmath, amssymb, amscd, paralist,tabularx,supertabular,amsmath, verbatim, amsthm, amssymb, mathrsfs, manfnt,times,latexsym,amscd,graphicx}
\usepackage[all]{xy}

\evensidemargin \oddsidemargin
\parskip=\medskipamount
\newcommand{\ZZ}{{\mathbb Z}}
\newcommand{\Q}{{\mathbb Q}}

\newcommand{\R}{{\mathcal R}}
\newcommand{\Sc}{{\mathcal S}}

\def\Gal{\operatorname{Gal}}

\def\ram{\operatorname{Ram}}

\def\mf{\mathfrak}

\def\bs{\backslash}

\newtheorem{thm}{Theorem}[section]
\newtheorem{theorem}[thm]{Theorem}

\newtheorem{lemma}[thm]{Lemma}	
\newtheorem{proposition}[thm]{Proposition}

\theoremstyle{definition}

\theoremstyle{remark}
\newtheorem{remark}[thm]{Remark}

\newtheorem*{lemma*}{Lemma}

\numberwithin{equation}{section}

\title{Isospectral towers of Riemannian manifolds}

\author{Benjamin Linowitz}
\address{Department of Mathematics\\ 
6188 Kemeny Hall\\
Dartmouth College\\
Hanover, NH 03755, USA}
\email[] {Benjamin.D.Linowitz@dartmouth.edu}

\begin{document}

\begin{abstract} 
In this paper we construct, for $n\geq 2$, arbitrarily large families of infinite towers of compact, orientable Riemannian $n$-manifolds which are isospectral but not isometric at each stage. In dimensions two and three, the towers produced consist of hyperbolic $2$-manifolds and hyperbolic $3$-manifolds, and in these cases we show that the isospectral towers do not arise from Sunada's method.
\end{abstract}

\maketitle

%%%%%%%%%%%%%%%%%%%%%
%%%%%%%%%%%%%%%%%%%%%
%%%%%%%%%%%%%%%%%%%%%
%%%%%%%%%%%%%%%%%%%%%
%%%%%%%%%%%%%%%%%%%%%
%%%%%%%%%%%%%%%%%%%%%
%%%%%%%%%%%%%%%%%%%%%
%%%%%%%%%%%%%%%%%%%%%

\section{Introduction}

Let $M$ be a closed Riemannian $n$-manifold. The eigenvalues of the Laplace-Beltrami operator acting on the space $L^2(M)$ form a discrete sequence of non-negative real numbers in which each value occurs with a finite multiplicity. This collection of eigenvalues is called the spectrum of $M$, and two Riemannian $n$-manifolds are said to be \textit{isospectral} if their spectra coincide. Inverse spectral geometry asks the extent to which the geometry and topology of $M$ is determined by its spectrum. Whereas volume and scalar curvature are spectral invariants, the isometry class is not. Although there is a long history of constructing Riemannian manifolds which are isospectral but not isometric, we restrict our discussion to those constructions most relevant to the present paper and refer the reader to \cite{Gordon-Survey} for an excellent survey.

In 1980 Vigneras \cite{vigneras} constructed, for every $n\geq 2$, pairs of isospectral but not isometric Riemannian $n$-manifolds. These manifolds were obtained as quotients of products $\textbf{H}_2^a\textbf{H}_3^{b}$ of hyperbolic upper-half planes and upper-half spaces by discrete groups of isometries obtained via orders in quaternion algebras. A few years later Sunada \cite{Sunada} described an  algebraic method of producing isospectral manifolds, which he used to construct further examples of isospectral but not isometric Riemann surfaces. Sunada's method is extremely versatile and is responsible for the majority of the known examples of isospectral but not isometric Riemannian manifolds.

Before stating our main result we require some additional terminology. We will say that a collection $\{ M_i \}_{i=0}^\infty$ of Riemannian manifolds is a tower if for every $i\geq 0$ there is a finite cover $M_{i+1}\rightarrow M_i$. We shall say that two infinite towers $\{ M_i\}$ and $\{N_i\}$ of Riemannian $n$-manifolds are \textit{isospectral towers} if for every $i\geq 0$, $M_i$ and $N_i$ are isospectral but not isometric.

\begin{theorem}\label{theorem:intromain}
For all $m,n\geq 2$ there exist infinitely many families of isospectral towers of compact Riemannian $n$-manifolds, each family being of cardinality $m$.
\end{theorem}

Like Vigneras' examples, our towers arise as quotients of products of hyperbolic upper-half planes and spaces so that in dimension two (respectively dimension three) the manifolds produced are hyperbolic $2$-manifolds (respectively hyperbolic $3$-manifolds). It follows from the Mostow Rigidity Theorem (see \cite{Mostow-Rigidity}, \cite[page 69]{mac-reid-book}) that in dimension three, the fundamental groups of the manifolds at each stage of the isospectral towers are non-isomorphic.

The proof of Theorem \ref{theorem:intromain} is entirely number-theoretic and relies on the construction of suitable chains of quaternion orders. To our knowledge the only construction of isospectral towers are given by McReynolds \cite{Mcreynolds}, who used a variant of Sunada's method to produce manifolds as quotients of symmetric spaces associated to non-compact Lie groups. Although McReynolds' methods apply in our setting, our results are of independent interest. Not only do we produce arbitrarily large families of isospectral towers, but we additionally show that in dimensions two and three our isospectral towers cannot arise from Sunada's method. 

\begin{theorem}\label{theorem:intronotsunada}
In dimensions two and three the isospectral towers constructed in Theorem \ref{theorem:intromain} do not arise from Sunada's method.
\end{theorem}

Chen \cite{Chen} has shown that Vigneras' two and three dimensional examples of isospectral but not isometric hyperbolic manifolds cannot arise from Sunada's method. Chen's proof relies on the fact that Vigneras' manifolds are obtained from maximal orders in quaternion algebras. We are able to extend Chen's argument to the non-maximal orders contained in our chains of quaternion orders by employing J\o rgensen  and Thurston's proof \cite{thurston,meyerhoff} that the set of volumes of hyperbolic $3$-orbifolds comprise a well-ordered subset of the non-negative real numbers.

It is a pleasure to thank Peter Doyle, Carolyn Gordon, Tom Shemanske and David Webb for interesting and valuable conversations regarding various aspects of this paper.

\section{Genera of quaternion orders}\label{section:quaternionalgebraprelims}

We begin with some facts concerning genera of quaternion orders. Most of these facts are well-known and can be found in \cite{vigneras,mac-reid-book}; for proofs of the others, see Section 3 of \cite{linowitz}.

Let $K$ be a number field and $B/K$ a quaternion algebra. We denote by $\ram(B)$ the set of primes at which $B$ ramifies, by $\ram_f(B)$ the set of finite primes of $\ram(B)$ and by $\ram_\infty(B)$ the set of archimedean primes of $\ram(B)$. If $R$ is a subring of $B$ then we denote by $R^1$ the multiplicative group consisting of the elements of $R$ having reduced norm one. For a prime $\nu$ of $K$ (finite or infinite) we denote by $B_\nu$ the quaternion algebra $B\otimes_K K_\nu$. Let $\R$ be an $\mathcal O_K$-order of $B$. If $\nu$ is a finite prime of $K$, denote by $\R_\nu$ the $\mathcal O_{K_\nu}$-order $\R\otimes_{\mathcal O_K} \mathcal{O}_{K_\nu}$ of $B_\nu$. 

Two orders $\R,\Sc$ of $B$ are said to lie in the same genus if $\R_\nu\cong\Sc_\nu$ for all finite primes $\nu$ of $K$. The genus of $\R$ is the disjoint union of isomorphism classes and it is known that the number of isomorphism classes in the genus of $\R$ is finite (and is in fact a power of $2$ whenever there exists an archimedean prime of $K$ not lying in $\ram_\infty(B)$). When $\R$ is a maximal order of $B$, the number of isomorphism classes in the genus of $\R$ is called the type number of $B$.

The genus of an order $\R$ is profitably characterized adelically. Let $J_K$ be the idele group of $K$ and $J_B$ the idele group of $B$. The genus of $\R$ is characterized by the coset space $J_B/\mathfrak N(\R)$ where $\mathfrak N(\R)=J_B\cap \prod_\nu N(\R_\nu)$ and $N(\R_\nu)$ is the normalizer of $\R_\nu$ in $B_\nu^\times$. The isomorphism classes in the genus of $\R$ then correspond to double cosets of $B^\times\bs J_B / \mathfrak N(\R)$. This correspondence is given as follows. Suppose that $\Sc$ is in the genus of $\R$ and let $\tilde{x}=(x_\nu)\in J_B$ be such that $x_\nu\Sc_\nu x_\nu^{-1}=\R_\nu$. Then the isomorphism class of $\Sc$ corresponds to the double coset $B^\times \tilde{x} \mathfrak N(\R)$. 

The reduced norm induces a bijection $n: B^\times\bs J_B / \mathfrak N(\R)\rightarrow K^\times\bs J_K / n(\mathfrak N(\R))$. Define $H_\R=K^\times n(\mathfrak N(\R))$ and $ G_\R=J_K / H_\R$. As $J_K$ is abelian there is an isomorphism $ G_\R\cong K^\times\bs J_K / n(\mathfrak N(\R))$. The group $H_\R$ is an open subgroup of $J_K$ having finite index, so that by class field theory there exists a class field $K(\R)$ associated to it. The extension $K(\R)/K$ is an abelian extension of exponent $2$ whose conductor is divisible only by the primes dividing the level ideal of $\R$ (that is, the primes $\nu$ such that $\R_\nu$ is not a maximal order of $B_\nu$) and whose Galois group is isomorphic to $G_\R$. In the case that $\R$ is maximal we will often write $K(B)$ in place of $K(\R)$. Note that $K(B)$ can be characterized as the maximal abelian extension of $K$ of exponent $2$ which is unramified outside the real places of $\ram(B)$ and in which all primes of $\ram_f(B)$ split completely \cite[page 39]{Chinburg-Friedman}.

Given orders $\Sc,\mathcal T$ lying in the genus of $\R$, define the $G_\R$\textit{-valued} \textit{distance idele} $\rho_{\R}(\Sc,\mathcal T)$ as follows. Let $\tilde{x}_\Sc,\tilde{x}_{\mathcal T}\in J_B$ be such that ${x_\Sc}_\nu \Sc_\nu {x_\Sc}_\nu^{-1}=\R_\nu$ and ${x_\mathcal T}_\nu \mathcal T_\nu {x_\mathcal T}_\nu^{-1}=\R_\nu$ for all $\nu$. We define $\rho_{\R}(\Sc,\mathcal T)$ to be the coset $n(\tilde{x}_\Sc^{-1}\tilde{x}_\mathcal T)H_\R$ in $G_\R$. The function $\rho_{\R}(-,-)$ is well-defined and has the property that $\rho_{\R}(\Sc,\mathcal T)$ is trivial if and only if $\Sc\cong \mathcal T$ \cite[Prop 3.6]{linowitz}.

\section{Arithmetic groups derived from quaternion algebras}\label{section:arithmeticgroups}

In this section we recall some facts about arithmetic groups derived from orders in quaternion algebras; see \cite{mac-reid-book} for proofs.

Let $K$ be a number field with $r_1$ real primes and $r_2$ complex primes so that $[K:\mathbb Q]=r_1+2r_2$, and denote by $V_\infty$ the set of archimedean primes of $K$. Let $B/K$ a quaternion algebra and consider the isomorphism $$B\otimes_\mathbb Q \mathbb R \cong \mathbb H^r \times M_2(\mathbb R)^s \times M_2(\mathbb C)^{r_2}\qquad (r+s=r_1).$$

There exists an embedding $$\rho: B^\times \hookrightarrow \prod_{\nu\in V_\infty \setminus\ram_{\infty}(B)} B_\nu^\times \longrightarrow G = GL_2(\mathbb R)^s \times GL_2(\mathbb C)^{r_2}.$$ Restricting $\rho$ to $B^1$ gives an embedding $\rho: B^1 \hookrightarrow G^1=SL_2(\mathbb R)^s \times SL_2(\mathbb C)^{r_2}$. Let $N$ be a maximal compact subgroup of $G^1$ so that $X:=G^1/N=\textbf{H}_2^s\textbf{H}_3^{r_2}$ is a product of two and three dimensional hyperbolic spaces.

If $\R$ is an order of $B$ then $\rho(\R^1)$ is a discrete subgroup of $G^1$ of finite covolume. The orbifold $\rho(\R^1)\bs X$ is compact if and only if $B$ is a division algebra and will be a manifold if $\rho(\R^1)$ contains no elliptic elements.

\section{Isospectral towers and chains of quaternion orders}\label{section:towersandchains}

As was the case with Vigneras' \cite{vigneras} examples of Riemannian manifolds which are isospectral but not isometric, our construction will make use of the theory of orders in quaternion algebras. In particular, our isospectral towers will arise from chains of quaternion orders.

\begin{theorem}\label{theorem:chainsexist}
Let $K$ be a number field and $B/K$ a quaternion algebra of type number $t\geq 2$. Then there exist infinitely many families of chains of quaternion orders $\{ \{\R_i^j\} : 1\leq j\leq t\}$ of $B$ satisfying:

	\begin{enumerate}
		\item For all $i\geq 0$ and $1\leq j\leq t$, $\R^j_0$ is a maximal order and $\R^j_{i+1}\subsetneq \R^j_i$.
		\item For all $i\geq 0$ and $1\leq j_1 < j_2\leq t$, $\R^{j_1}_i$ and $\R^{j_2}_i$ are in the same genus but are not conjugate.
		\item For all $i\geq 0$ and $1\leq j\leq t$, there is an equality of class fields $K(\R^j_i)=K(B)$.
	\end{enumerate}

\end{theorem}
	
Before proving Theorem \ref{theorem:chainsexist}, we show that it implies our main result. We will make use of the following proposition.

\begin{proposition}\label{proposition:noselectivity}
Suppose that $\ram_f(B)$ is not empty and that there is an archimedean prime of $K$ not lying in $\ram_\infty(B)$. Let $L$ be a maximal subfield of $B$ and $\Omega\subset L$ a quadratic $\mathcal O_K$-order. If $\{ \{\R_i^j\} : 1\leq j\leq t\}$ is a family of chains of quaternion orders of $B$ satisfying conditions (1)-(3) of Theorem \ref{theorem:chainsexist}, then for every $i\geq 0$ and $1\leq j_1<j_2\leq t$, $\R_i^{j_1}$ admits an embedding of $\Omega$ if and only if $\R^{j_2}_i$ admits an embedding of $\Omega$.
\end{proposition}
\begin{proof}
Our hypothesis that $\ram_f(B)$ be non-empty means that there exists a finite prime of $K$ which ramifies in $B$. Such a prime splits completely in $K(B)/K$ and thus in every subfield of $K(B)$ as well. It follows that $L\not\subset K(B)$, for otherwise would contradict the Albert-Brauer-Hasse-Noether theorem. Suppose now that there exists an $i\geq 0$ and $1\leq j_1\leq t$  such that $\R_i^{j_1}$ admits an embedding of $\Omega$. By condition (3), $L\not\subset K(\R_i)$, hence every order in the genus of $\R^{j_1}_i$ admits an embedding of $\Omega$ \cite[Theorem 5.7]{linowitz}. In particular, $\R^{j_2}_i$ admits an embedding of $\Omega$ for all $1\leq j_2\leq t$.\end{proof}

\begin{theorem}\label{theorem:main}
For all $m,n\geq 2$ there exist infinitely many families of isospectral towers of compact Riemannian $n$-manifolds, each family being of cardinality $m$.
\end{theorem}
\begin{proof}
Fix integers $m,n\geq 2$. Vigneras \cite[Theorem 8]{vigneras-isospectral} proved the existence of a number field $K$ and a quaternion algebra $B/K$ with the following properties:
\begin{itemize}
	\item $\ram_f(B)$ is nonempty
	\item The type number of $B$ is at least $2$.
	\item The only $\Q$-automorphisms of $B$ are inner automorphisms.
	\item If $\mathcal O$ is an order in $B$ then $\rho(\mathcal O^1)\backslash \textbf{H}_2^s\textbf{H}_3^{r_2}$ is a Riemannian $n$-manifold.
\end{itemize}

The proof of Vigneras' theorem is easily modified to construct quaternion algebras $B$ having arbitrarily large type number. In particular we will assume that $B$ satisfies the above properties and has type number greater than or equal to $m$. Because the type number of $B$ is a power of $2$, we may assume without loss of generality that the type number $t$ of $B$ is equal to $m$ and write $t=m=2^\ell$.

Let $\{ \{\R_i^j\} : 1\leq j\leq t\}$ be a family of chains of quaternion orders as in Theorem \ref{theorem:chainsexist}. For $1\leq j\leq t$, denote by $\{ M^j_i\}$ the infinite tower of finite covers of Riemannian $n$-manifolds associated to the $\{ \R^j_i\}$. By Theorem 4.3.5 of \cite{vigneras} and the discussion immediately afterwords, that condition (2) of Theorem \ref{theorem:chainsexist} is satisfied implies that $M^{j_1}_i$ and $M^{j_2}_i$ are not isometric for any $i\geq 0$ and $1\leq j_1<j_2\leq t$. 

It remains only to show that for all $i\geq 0$ and $1\leq j_1<j_2\leq t$, the manifolds $M^{j_1}_i$ and $M^{j_2}_i$ are isospectral. To ease notation set $\mathcal O_1=\R^{j_1}_i$ and $\mathcal O_2=\R^{j_2}_i$. By \cite[Theorem 6]{vigneras-isospectral}, it suffices to show that $\mathcal O_1^1$ and $\mathcal O_2^1$ have the same number of conjugacy classes of elements of a given reduced trace. If $h\in\mathcal O_1^1$, then the number of conjugacy classes of elements of $\mathcal O_1$ with reduced trace $\mbox{tr}(h)$ is equal to the number of embeddings of $\mathcal O_K[h]$ into $\mathcal O_1$ modulo the action of $\mathcal O_1^1$. It is well known that this quantity, when non-zero, is independent of the choice of order in the genus of $\mathcal O_1$. This can be proven as follows. Let $\Omega$ be a quadratic $\mathcal O_K$-order in a maximal subfield $L$ of $B$ and note that every embedding of $\Omega$ into $\mathcal O_1$ extends to an optimal embedding of some overorder $\Omega^*$ of $\Omega$ into $\mathcal O_1$. Conversely, if $\Omega^*$ is an overorder of $\Omega$, then every optimal embedding of $\Omega^*$ into $\mathcal O_1$ restricts to an embedding of $\Omega$ into $\mathcal O_1$. It therefore suffices to show that if $\mathcal O_2$ admits an embedding of $\Omega$, then the number of optimal embeddings of $\Omega$ into $\mathcal O_1$ modulo $\mathcal O_1^1$ is equal to the number of optimal embeddings of $\Omega$ into $\mathcal O_2$ modulo $\mathcal O_2^1$. This can be proven adelically using the same proof that Vigneras employed in the case of Eichler orders \cite[Theorem 3.5.15]{vigneras}.

By the previous paragraph, in order to prove that $M^{j_1}_i$ and $M^{j_2}_i$ are isospectral, it suffices to prove that if $\Omega$ is a quadratic order in a maximal subfield $L$ of $B$, then $\Omega$ embeds into $\R^{j_1}_i$ if and only if $\Omega$ embeds into $\R_i^{j_2}$. This follows from Proposition \ref{proposition:noselectivity}.\end{proof}

\section{Proof of Theorem \ref{theorem:chainsexist}}

\subsection{Orders in split quaternion algebras over local fields}
We begin by defining a family of orders in the quaternion algebra $M_2(k)$, $k$ a non-dyadic local field.

Let $k$ be a non-dyadic local field with ring of integers $\mathcal O_k$ and maximal ideal $P_k=\pi_k\mathcal{O}_k$. Let $F/k$ be the unramified quadratic field extension of $k$ with ring of integers $\mathcal O_F$ and maximal ideal $P_F=\pi_F\mathcal O_F$. We may choose $\pi_F$ so that $\pi_F=\pi_k$.

Let $\mathfrak B=M_2(k)$ and $\xi\in \mf B$ be such that $\mathfrak B = F + \xi F$ and $x\xi=\xi \overline x$ for all $x\in F.$ Let $m\geq 0$. Following Jun \cite{jun-primitive}, we define a family of orders in $\mathfrak B$: 

\begin{equation}\label{equation:definitionofjunorders}
R_{2m}(F)=\mathcal O_F + \xi \pi_F^m \mathcal O_F.
\end{equation}

For convenience we will denote $R_{2m}(F)$ by $R_{2m}$. It is easy to see that $R_0$ is a maximal order of $\mathfrak B$ and that we have inclusions: $$\cdots\subsetneq R_{2m+2}\subsetneq R_{2m}\subsetneq\cdots \subsetneq R_2\subsetneq R_0.$$

\begin{proposition}\label{proposition:normnormprop}
If $m\geq 1$, then $n(N(R_{2m}))=\mathcal{O}_k^\times {k^\times}^2$, where $N(R_{2m})$ is the normalizer of $R_{2m}$ in $\mathfrak B^\times$.\end{proposition}
\begin{proof} Suppose that $x\in N(R_{2m})$. Then $x\mathcal O_F x^{-1}\subseteq R_{2m}\subseteq R_0$ because $\mathcal O_F$ is contained in $R_{2m}$. Similarly, $x \pi_F^m \xi \mathcal O_F x^{-1}\subseteq R_{2m}$. Because $\pi_F=\pi_k$ is in the center of $\mathfrak B$, it follows that $x\xi\mathcal O_F x^{-1}\subseteq \pi_F^{-m}R_{2m}=\pi_F^{-m}\mathcal O_F + \xi\mathcal{O}_F$. 

We may write an arbitrary element of $x\xi\mathcal O_F x^{-1}$ as $\pi_F^{-m}a+\xi b$, where $a,b\in\mathcal O_F$. It is straightforward to show that $n(\pi_F^{-m}a+\xi b)=n(\pi_F^{-m}a)-n(b).$ As $b$ is an element of $\mathcal{O}_F$ we see that $n(b)$ is integral. Similarly, $n(\pi_F^{-m}a+\xi b)$ must be integral. This follows from observing that $\pi_F^{-m}a+\xi b\in x \xi \mathcal O_F x^{-1}$ is integral since every element of $\xi\mathcal O_F$ is integral and conjugation preserves integrality. Therefore $n(\pi_F^{-m}a)$ is integral, hence $a\in P_F^m$. We have shown that every element of $x\xi\mathcal O_F x^{-1}$ lies in $\mathcal O_F + \xi \mathcal O_F = R_0$. 
	
If $x\in N(R_{2m})$ then $x\mathcal O_F x^{-1}, x\xi\mathcal O_F x^{-1}\subseteq R_0$. Therefore $x R_0x^{-1}=x\left( \mathcal O_F + \xi\mathcal O_F\right)x^{-1}\subseteq R_0.$ Equivalently, $x\in N(R_0)$ hence $N(R_{2m})\subseteq N(R_0)$. Recalling that $R_0$ is conjugate to $M_2(\mathcal O_k)$, whose normalizer is $GL_2(\mathcal O_k) k^\times$, we deduce that $n(N(R_{2m}))\subseteq \mathcal O_k^\times {k^\times}^2$. To show the reverse inclusion we note that both $\mathcal O_F^\times$ and $k^\times$ lie in $N(R_{2m})$ and that $n(\mathcal O_F^\times k^\times)=\mathcal O_k^\times {k^\times}^2$ (since $F/k$ is unramified).\end{proof}

\begin{proposition}\label{proposition:norm1}
If $\mathcal O=\mathcal O_k[R_{2m}^1]$ is the ring generated over $\mathcal O_k$ by $R_{2m}^1$, then $\mathcal O=R_{2m}$.
\end{proposition}
\begin{proof}
Consider the map $f: R_{2m}^\times\rightarrow \mathcal O_{k}^\times/{\mathcal O_{k}^\times}^2$ induced by composing the reduced norm on $R_{2m}^\times$ with the projection $\mathcal O_{k}^\times\rightarrow\mathcal O_{k}^\times/{\mathcal O_{k}^\times}^2$. This map is surjective as $n(\mathcal O_F^\times)=\mathcal O_k^\times$ and $\mathcal O_F\subset R_{2m}$. The kernel of $f$ is the disjoint union of the cosets $g R_{2m}^1$ as $g$ varies over $\mathcal O_k^\times$. Because $k$ is a non-dyadic local field, $\mathcal O_{k}^\times/{\mathcal O_{k}^\times}^2$ has order $2$. As $\mathcal O^\times$ contains the kernel of $f$ as well as elements of non-square reduced norm (this follows from the fact that $\mathcal O^\times$ contains a $k$-basis for $\mf B$. Fix an element of $\mf B$ of non-square reduced norm and write it as a $k$-linear combination of this basis, then clear denominators by multiplying through by an element of $\mathcal O_k^2$. The resulting element will lie in $\mathcal O^\times$ and have a non-square reduced norm), $\mathcal O^\times=R_{2m}^\times$. Using equation (\ref{equation:definitionofjunorders}) and the observation that one may take as $\xi$ the matrix with zeros on the diagonal and ones on the anti-diagonal, it is easy to see that $R_{2m}$ has an $\mathcal O_k$-basis consisting entirely of elements of $R_{2m}^\times$, hence $R_{2m}\subset \mathcal O$. As the opposite inclusion is immediate, we conclude that $\mathcal O=R_{2m}$.
\end{proof}

\subsection{Proof of Theorem \ref{theorem:chainsexist}}\label{subsection:theconstruction}

Let notation be as above and assume that $B$ has type number $t=2^\ell$ with $\ell\geq 1$. Let $\R$ be a fixed maximal order of $B$. We will prove Theorem \ref{theorem:chainsexist} in the case that $\ell=2$ as the other cases are similar though in general more tedious in terms of notation.

Because the type number of $B$ is $4$ there is an isomorphism $\Gal(K(B)/K)\cong \ZZ/2\ZZ \times \ZZ/2\ZZ$. Let $\mu_1,\mu_2$ be primes of $K$ such that the Frobenius elements $(\mu_1, K(B)/K)$ and $(\mu_2, K(B)/K)$ generate $\Gal(K(B)/K)$. By the Chebotarev density theorem we may choose $\mu_1$ and $\mu_2$ from sets of primes of $K$ having positive density.

Let $\delta_i=\mbox{diag}(\pi_{K_{\mu_i}},1)$ and $F_i$ be the unramified quadratic field extension of $K_{\mu_i}$ (for $i=1,2$).

For any $i\geq 0$ and $0\leq a,b\leq 1$ define the orders $\R_i^{a,b}$ via:

\begin{displaymath}
({\R_i^{a,b}})_\nu = \left\{ \begin{array}{ll}
\R_\nu & \textrm{if $\nu\not\in \{\mu_1,\mu_2\}$}\\
\delta_1^a R_{2i}(F_1) \delta_1^{-a} & \textrm{if $\nu=\mu_1$}\\
\delta_2^b R_{2i}(F_2)\delta_2^{-b} & \textrm{if $\nu=\mu_2$}\\
\end{array} \right.
\end{displaymath}

We now set $\R_i^1=\R_i^{0,0},\R_i^2=\R_i^{1,0},\R_i^3=\R_i^{0,1}$ and $\R_i^4=\R_i^{1,1}$. 

Having constructed our chains $\{\R_i^1\},\{\R_i^2\},\{\R_i^3\},\{\R_i^4\}$, we show that they satisfy the required properties. That $\R^j_0$ is a maximal order and $\R^j_{i+1}\subsetneq \R^j_i$ for all $i\geq 0$ and $1\leq j\leq 4$ is immediate. We claim that $K(\R^j_i)=K(B)$ for all $i\geq 0$ and $1\leq j\leq 4$. Indeed, as $n(N({\R^j_i}_\nu))=n(N(\R_\nu))$ for all primes $\nu$ of $K$ (by Proposition \ref{proposition:normnormprop}), it is immediate that $H_{\R^j_i} = K^\times n(\mathfrak N(\R^j_i))$ is equal to $H_\R = K^\times n(\mathfrak N(\R))$. It follows that $K(\R^j_i)=K(\R)$ and because $K(\R)=K(B)$ by definition, we have proven our claim.

By construction $\R_i^1,\R_i^2,...,\R_i^4$ all lie in the same genus. That they represent distinct isomorphism classes can be proven by considering the $G_{\R^1_i}$-valued distance idele $\rho_{\R^1_i}(-,-)$. For instance, to show that $\R_i^1\not\cong\R_i^2$ note that $\rho_{\R_i^1}(\R_i^1,\R_i^2)=n(\tilde{x}_{\R^2_i})H_{\R^1_i}=e_{\mu_1} H_{\R^1_i}$, where $e_{\mu_1}=(1,...,1,\pi_{K_{\mu_1}},1,...)\in J_K$. This corresponds, under the Artin reciprocity map, to the element $(\mu_1, K(B)/K)\in\mbox{Gal}(K(B)/K)$, which is non-trivial by choice of $\mu_1$. This shows that properties (1)-(3) are satisfied, concluding our proof.

\section{The Sunada construction}\label{section:notsunada}

In this section we prove that in dimensions two and three the isospectral towers constructed in Theorem \ref{theorem:main} do not arise from Sunada's method \cite{Sunada}. Our proof follows Chen's proof \cite{Chen} that the two and three dimensional examples of isospectral but not isometric Riemannian manifolds that Vigneras constructed in \cite{vigneras-isospectral} do not arise from Sunada's construction, though several difficulties arise from our introduction of non-maximal quaternion orders.

We begin by reviewing Sunada's construction. Let $H$ be a finite group with subgroups $H_1$ and $H_2$. We say that $H_1$ and $H_2$ are \textit{almost conjugate} if for every $h\in H$,  $$\#([h]\cap H_1)=\#([h]\cap H_2), $$ where $[h]$ denotes the $H$-conjugacy class of $h$. In this case we call $(H,H_1,H_2)$ a \textit{Sunada triple}. 

Let $M_1$ and $M_2$ be Riemannian manifolds and $(H,H_1,H_2)$ a Sunada triple. If $M$ is a Riemannian manifold then as in Chen \cite{Chen} we say that $M_1$ and $M_2$ are \textit{sandwiched} between $M$ and $M/H$ with triple $(H,H_1,H_2)$ if there is an embedding of $H$ into the group of isometries of $M$ such that the projections $M\rightarrow M / H_i$ are Riemannian coverings and $M_i$ is isometric to $M / H_i$ (for $i=1,2$). The following theorem of Sunada \cite{Sunada} provides a versatile means of producing isospectral Riemannian manifolds.  

\begin{theorem}[Sunada]
If $M_1$ and $M_2$ are Riemannian manifolds which can be sandwiched with a triple $(H,H_1,H_2)$ then $M_1$ and $M_2$ are isospectral.
\end{theorem}

The majority of the known examples of isospectral but not isometric Riemannian manifolds have been constructed by means of Sunada's theorem and its variants (see \cite{Gordon-SunadaSurvey} for a survey of work inspired by Sunada's theorem). We will show that in dimensions two and three, none of the manifolds produced in Theorem \ref{theorem:main} arise from Sunada's construction.

Let $\{ M_i \}, \{ N_i \}$ be isospectral towers of Riemannian $n$-manifolds as constructed in Theorem \ref{theorem:main}. Recall that these towers were constructed from chains $\{ \R_i\}$, $\{\Sc_i\}$ of  orders in a quaternion algebra $B$ defined over a number field $K$. These towers will consist of hyperbolic $2$-manifolds or hyperbolic $3$-manifolds if and only if $B$ is unramified at a unique archimedean prime of $K$.

\begin{theorem}\label{theorem:notsunada}
If $B$ is unramified at a unique archimedean prime of $K$ then for every $i\geq 0$ the isospectral hyperbolic manifolds $M_i$ and $N_i$ do not arise from Sunada's construction.
\end{theorem}

Our proof of Theorem \ref{theorem:notsunada} will require the following lemma.

\begin{lemma}\label{lemma:lemma1}
Let $\mathcal O\in \{\R_i : i\geq 0 \}$ and $ J\supseteq \mathcal O^1$ be a subgroup of $B^1$. If $[J:\mathcal O^1]<\infty$ then $J\subseteq \R^1_0$.
\end{lemma}

\begin{proof}
Denote by $\mathcal J=\mathcal O_K[J]$ the ring generated over $\mathcal O_K$ by $J$. Write $J=\bigcup_{j=1}^\infty\{g_j\mathcal O^1\}$ and $\mathcal J=\sum\mathcal O_K\{g_j\mathcal O^1\}$. As $[J:\mathcal O^1]<\infty$, $\mathcal J$ is a finitely generated $\mathcal O_K$-module containing $\mathcal O_k[\mathcal O^1 ]$ and hence an order of $B$. Let $\mathcal M$ be a maximal order of $B$ containing $\mathcal J$. Then $\mathcal O^1\subseteq \mathcal M\cap\R_0$, hence $\mathcal O_\nu\subseteq \mathcal M_\nu\cap{\R_0}_\nu$ for all primes $\nu$ of $K$. Recall from Section \ref{subsection:theconstruction} that the chain $\{ \R_i\}$ was constructed using a finite set of primes $\{ \mu_1,..,\mu_t\}$ of $K$ so that for every $i\geq 0$, ${\R_i}_\nu$ is maximal if $\nu\not\in \{ \mu_1,..,\mu_t\}$ and ${\R_i}_{\mu_j}$ is conjugate to $R_{2i}(F_j)$ for $j=1,..,t$. We will prove the Lemma in the case that ${\R_i}_{\mu_j}=R_{2i}(F_j)$ for $j=1,..,t$; the other cases follow from virtually identical proofs upon appropriately conjugating every appearance of $R_{2i}(F_j)$.

If $\nu\not\in \{ \mu_1,..,\mu_t\}$ then $\mathcal M_\nu={\R_0}_\nu=\mathcal O_\nu$. At a prime $\mu\in\{ \mu_1,..,\mu_t\}$ we see that $\mathcal O_\mu^1\subseteq \mathcal M_\mu\cap {\R_0}_\mu$, hence $\mathcal O_\mu \subseteq \mathcal M_\mu \cap {\R_0}_\mu$ by Proposition \ref{proposition:norm1}. Therefore $\mathcal M_\mu$ contains $\mathcal O_F$ (since $\mathcal O_F\subset R_{2i}(F)=\mathcal O_\mu$). By Proposition 2.2 of \cite{jun-primitive} there is a unique maximal order of $B_\mu$ containing $\mathcal O_F$, hence $\mathcal M_\mu={\R_0}_\mu$. We have shown that for all $\nu$, $\mathcal M_\nu={\R_0}_\nu$, hence $\mathcal M=\R_0$. As $J\subset \mathcal J^1\subseteq \mathcal M^1=\R_0^1$, we have proven the lemma.\end{proof}

\begin{remark}\label{remark:lemma1remark}
By conjugating every appearance of $\R_i$ and $\R_0$ in the proof of Lemma \ref{lemma:lemma1} one obtains the following: Let $x\in B^\times$, $\mathcal O\in \{x\R_i x^{-1} : i\geq 0 \}$ and $J\supseteq \mathcal O^1$ be a subgroup of $B^1$. If $[J:\mathcal O^1]<\infty$ then $J\subseteq x \R^1_0 x^{-1}$.
\end{remark}

We now prove Theorem \ref{theorem:notsunada}. Write $M_i=\rho(\R_i^1)\backslash X$ and  $N_i=\rho(\Sc_i^1)\backslash X$, where $X$ is either two or three dimensional hyperbolic space (depending on whether the archimedean prime of $K$ which is unramified in $B$ is real or complex). Finally, suppose that $M_i$ and $N_i$ arose from Sunada's construction. That is, there exists a Riemannian manifold $M$ and a Sunada triple  $(H, H_1, H_2)$ such that $M_i$ and $N_i$ are sandwiched between $M$ and $M/H$. Let $\Gamma$ be the discrete subgroup of $G = GL_2(\mathbb R)^s \times GL_2(\mathbb C)^{r_2}$ such that $M/H=\Gamma\backslash X$.   
 
Recall that the set of volumes of hyperbolic 3-orbifolds comprise a well-ordered subset of the real numbers \cite{thurston, gromov, meyerhoff} and that an analogous statement holds for hyperbolic 2-orbifolds by the Riemann-Hurwitz formula. We may therefore choose the Sunada triple $(H, H_1, H_2)$ and the Riemannian manifold $M$ so that the following property is satisfied: If $(H^\prime, H_1^\prime, H_2^\prime)$ is a Sunada triple and $M^\prime$ a Riemannian manifold such that $M_i$ and $N_i$ are sandwiched between $M^\prime$ and $M^\prime/H^\prime$, then writing $M^\prime/H^\prime=\Gamma^\prime\backslash X$ we have that $\mbox{CoVol}(\Gamma^\prime\cap \rho(B^1))\geq \mbox{CoVol}(\Gamma\cap \rho(B^1))$.

We have the following diagram of coverings:

$$\xymatrix@R=10pt@C=15pt{
&&M\ar@{=}[d]\\
&&{\Gamma_0\backslash X}\ar@{->}[ddll]\ar@{->}[ddrr]\ar@{->}[dddd]\\ \\
**[l]\rho(\R_i^1)\bs X\cong \Gamma_1^\prime\bs X  \ar@{->}[ddrr]  &&&& **[r]{\Gamma_2^\prime\bs X \cong \rho(\Sc_i^1)\bs X}\ar@{->}[ddll]\\ \\
&&{\Gamma\bs X}\ar@{=}[d]\\
&&M/H
}$$

where $\Gamma_0\subseteq \Gamma_1^\prime, \Gamma_2^\prime \subseteq \Gamma$ are discrete subgroups of $G$. We have $\Gamma_1^\prime=\gamma_1\rho(\R_i^1)\gamma_1^{-1}$ and $\Gamma_2^\prime=\gamma_2\rho(\Sc_i^1)\gamma_2^{-1}$ for some $\gamma_1,\gamma_2\in G^\times$. By conjugating if necessary, we may assume that $\gamma_1=\mbox{id}$. Since $\Gamma_0$ is of finite index in both $\rho(\R_i^1)$ and $\gamma_2\rho(\Sc_i^1)\gamma_2^{-1}$, by Lemma 4 of \cite{Chen} there is an element $g\in\rho(B^\times)$ such that $\gamma_2\rho(\Sc_i^1)\gamma_2^{-1}=g\rho(\Sc_i^1)g^{-1}$. If $\Gamma$ is of finite index over $\rho(\R_i^1)$ then $\Gamma\cap \rho(B^1)$ is of finite index over $\rho(\R_i^1)$ and by Lemma \ref{lemma:lemma1}, $\Gamma\cap \rho(B^1)\subseteq \rho(\R_0^1)$. Similarly, if $\Gamma$ is of finite index over $g\rho(\Sc_i^1)g^{-1}$ then Remark \ref{remark:lemma1remark} shows that $\Gamma\cap \rho(B^1)\subseteq g\rho(\Sc_0^1)g^{-1}$.
This shows that $\rho(\R_i^1)$ and $g\rho(\Sc_i^1)g^{-1}$ are almost conjugate subgroups of both $\rho(\R_0^1)$ and $g\rho(\Sc_0^1)g^{-1}$ (they cannot be conjugate in either group because $M_i$ and $N_i$ are not isometric). Because we have chosen $(H,H_1,H_2)$ so that the covolume of $\Gamma\cap\rho(B^1)$ is minimal, it must be the case that $\rho(\R_0^1)=\Gamma\cap\rho(B^1)=g\rho(\Sc_0^1)g^{-1}$. This is a contradiction as $M_0=\rho(\R_0^1)\backslash X$ and $N_0=\rho(\Sc_0^1) \backslash X$ are not isometric.

Thus $\Gamma$ is not of finite index over $\rho(\R_i^1)$, hence the cover $M=\Gamma_0\backslash X \rightarrow \Gamma\backslash X=M/H$ is not finite, a contradiction. This finishes the proof of Theorem \ref{theorem:notsunada}.

\end{document}